\documentclass[11pt]{amsart}

\usepackage{amsmath, a4wide}
\usepackage{amssymb,amsthm,amsfonts}
\usepackage{amsrefs}
\usepackage[utf8]{inputenc}
\usepackage[T1]{fontenc}
\usepackage{xcolor}
\usepackage[colorlinks]{hyperref}

\newcommand{\Z}{{\mathbb Z}}
\newcommand{\N}{{\mathbb N}}
 
\newcommand{\R}{{\mathbb R}}
\newcommand{\T}{{\mathbb T}}
\newcommand{\p}{{\phi}}

\newcommand{\Per}{\mathrm{Per}}

\newtheorem{theorem}{Theorem}[section]
\newtheorem{corollary}[theorem]{Corollary}

\newtheorem{lemma}[theorem]{Lemma}
\newtheorem{proposition}[theorem]{Proposition}
\newtheorem{assumption}{Assumption}

\theoremstyle{definition}
\newtheorem{definition}[theorem]{Definition}
\newtheorem{remark}[theorem]{Remark}

\newcommand{\eps}{\varepsilon}

\begin{document}
 
\title{Periodic partitions with minimal perimeter}

\author{Annalisa Cesaroni}
\address{Dipartimento di Matematica, Universit\`{a} di Padova, Via Trieste 63, 35131 Padova, Italy}
\email{annalisa.cesaroni@unipd.it}
\author{Matteo Novaga}
\address{Dipartimento di Matematica, Universit\`{a} di Pisa, Largo Bruno Pontecorvo 5, 56127 Pisa, Italy}
\email{matteo.novaga@unipi.it}

 \begin{abstract} 
We show existence of  fundamental domains which minimize a general perimeter functional
in a homogeneous metric measure space. 
In some cases, which include the usual perimeter in the universal cover of a closed Riemannian  manifold, and 
 the fractional perimeter in $\R^n$, we can prove regularity of the minimal domains.
As a byproduct of our analysis we obtain that a countable partition which is minimal for the fractional perimeter is locally finite and regular,
extending a result previously known for the local perimeter.
Finally,  in the planar case we provide a detailed description of  the fundamental domains which are minimal for a general anisotropic perimeter.
\end{abstract} 
  
 
\maketitle 
 
\tableofcontents
 
\section{Introduction}
In this paper we  deal with fundamental domains of finite perimeter, for a very general notion of perimeter functional,
in a homogeneous metric measure space $(X, \mu, d)$ equipped with a group $G$ of measure-preserving homeomorphisms,
and we look  for fundamental domains with minimal  perimeter,
which we call {\it isoperimetric fundamental domains}. A typical example of such space is the universal cover of a closed Riemannian manifold $M$, with the usual notion of surface area.

This question of basic interest has been already considered in the literature. 
In particular, in \cite{choe} the author proved  existence  and partial regularity of isoperimetric fundamental domains
of a closed Riemannian $3$-manifold $M$, with respect to the classical perimeter functional.  
If $M$ is irreducible,  i.e.  every embedded sphere in $M$ bounds a ball, he 
also showed existence of a fundamental domain of minimal perimeter among the class  of fundamental domains whose interior is  homeomorphic to a ball 
(the projection on $M$ of the boundary of such domains is usually called a {\it spine}), using the theory of integral varifolds.

In the $2$-dimensional case the situation is relatively well understood: 
every fundamental domain
 with least boundary length is 
homeomorphic to a disk, its boundary consists of geodesic segments meeting each other
at angles of $2\pi/3$, and the number of such segments is $6-6\chi(M)$ (see \cite[Section 4]{choe} and \cite[Theorem 1.2]{mnpr}). 
The paper \cite{mnpr} also contains a detailed study of  minimal  spines, which are  critical points of the length functional.  
Finally, in \cite{h} a similar isoperimetric problem has been  considered in the case of  $2$-dimensional tori, restricting the class of domains to centrally symmetric convex sets. In this case the author shows that the only minimizers are hexagons and parallelograms.

Here we shall consider this isoperimetric problem in greater generality. We introduce a notion of perimeter which is sufficiently general 
to include the classical  local perimeter (isotropic and anisotropic) and the nonlocal perimeter of fractional type, and we provide existence 
of a fundamental domain of minimal perimeter by exploiting a by now standard procedure based on 
lower semicontinuity and compactness properties of the perimeter functional, together with a concentration compactness argument. 
This argument dates  back to Almgren \cite{alm} and has been used several times in isoperimetric problems in order to deal  with the possible  loss of 
mass at infinity of minimizing sequences. 

Let us point out that the existence of an isoperimetric fundamental domain is equivalent to the existence of a minimal partition of the space $X$ among all partitions which are invariant  with respect to the action of the group $G$. Exploiting this fact, we then pass to the analysis of the regularity property of 
 minimal fundamental domains. We reduce first of all to the case  in which  either $X$  is the universal 
 covering of a closed Riemannian manifold, and $\Per$ is the perimeter functional on $M$ 
 associated to a given norm (and then lifted to $X$) or $X=\R^n$ with the group of integer translations 
 and $\Per$ is either the local  or the fractional perimeter.  
 First of all we observe that   the minimal partitions among $G$-periodic partitions are also $(\Lambda,r)$-mininizer 
 of the perimeter functional for $r$ smaller than the injectivity radius of $M$ (see Definition \ref{defminimi} and Proposition \ref{partmin}).

The first regularity result that we prove is the boundedness of every isoperimetric fundamental domain $D$. This property 
  is equivalent to the local finiteness of the $G$-periodic partition of $X$ generated by $D$, see Proposition \ref{locallyfinite}. This result, in the   case of countable partitions 
and local classical perimeter, is due to \cite{mt} in the context of variational  image segmentation problems  (see also the previous paper \cites{ct1,ct2}).  
First of all, the local finiteness of every locally minimal conical partition is proved, 
by using the Elimination Lemma \ref{infiltration},   and a dimensional  reduction procedure, see Proposition  \ref{conicalpart}. Then, the case of general $(\Lambda,r)$-minimal partitions
is obtained by applying a blow-up procedure, then by showing that the blow-up of a minimal partition is a conical minimal partition, and finally by concluding again with the Elimination Lemma. 
To obtain that the blow-up of a partition is a conical partition, a monotonicity formula is used: this has been provided for countable partitions and the classical perimeter
 functional in \cite[Lemma 5]{mt}, whereas in for finite partitions and the fractional perimeter has been obtained in \cite[Theorem 3,10]{cm} 
 by passing to the extension problem. In this paper we adapt the previous result also to the case of countable partitions. In particular, as a byproduct, we have that countable partitions which are $(\Lambda,r)$- minimal for the fractional perimeter are locally finite.
  
 Once that we get the boundedness of the fundamental domain $D$, we exploit  
 regularity results for finite partitions (obtained for the classical perimeter in \cite{mt} and for the fractional perimeter in \cite{cm}) concluding that the boundary of $D$ is a smooth hypersurface, 
 up to a  nonempty closed singular set of Hausdorff dimension at most $n-2$ (discrete for $n=2$).  
 As a byproduct of this result, we  get  that actually there exist  minimal cones
(with more than two phases)  for the fractional perimeter.  We recall that the authors in \cite{cn} showed  that, 
 when the fractional order of the perimeter is sufficiently closed to $1$, the planar $3$-cone with angles of $2\pi/3$ is locally minimal. 

We observe that these arguments do not apply directly to the case of the anisotropic perimeter, due to the absence  of a monotonicity formula; 
in particular we cannot prove the boundedness of 
isoperimetric fundamental domains in any dimension, but only in the case of the plane, see Proposition \ref{probounded}. 
 Nevertheless, for sufficiently regular anisotropies (that is uniformly convex $C^2$ anisotropies) we get that the boundary of every isoperimetric fundamental domain is  
 up to closed singular set of Hausdorff dimension at most $n-2$ (discrete for $n=2$), a  $C^{1,\alpha}$ hypersurface (or more regular if the anisotropy is more regular).
Finally, for the case of  anisotropic homogeneous perimeter in the plane, we  show that for strictly 
convex anisotropies the only fundamental domains are hexagons and parallelograms, and moreover, that if the anisotropy is also differentiable then parallelograms cannot be minimizers. 

We conclude by mentioning that  we do not discuss here the interesting question if the isoperimetric fundamental domains have interior homeomorphic to a ball, or more generally if  they are contractible. In the case of planar flat torus with the local (possibly anisotropic) perimeter, the answer is affermative since  isoperimetric  domains are hexagons or a parallelograms.
In the $3$-dimensional case, with the usual perimeter, Lord  Kelvin proposed in \cite{thompson} (see also \cite{kelvin}) an explicit candidate which is homeomorphic to a ball, but it is still an open question whether such candidate is actually a minimizer. 

\smallskip
The plan of the paper is the following: in Section \ref{secnot} we introduce the notion of perimeter, fundamental domain and $G$-periodic partition. 
In Section \ref{secex} we prove existence of a minimal $G$-periodic partition or, equivalently, of an isoperimetric fundamental domain.
In Section \ref{secloc} we consider the case of $X=\R^n$ and $G=\Z^n$, and we prove partial regularity of minimal partitions when the perimeter functional is the local perimeter or the fractional perimeter.
Eventually, in Section \ref{secplane} we discuss in detail the case of the anisotropic perimeter in the plane,
showing that a minimal partition is given by a locally finite Steiner network, and the isoperimetric fundamental domain is a centrally symmetric convex hexagon or parallelogram.

\smallskip

{\it Acknowledgements.} The authors are member of INDAM-GNAMPA; the second author was supported by the PRIN Project 2019/24. 

 \section{Notation and setting} \label{secnot}

Let $(X, d, \mu)$  to be a locally compact complete metric measure space, equipped with a distance $d$ and a $\sigma$-finite Radon measure $\mu$, with  $\mu(X)>0$.  We denote by $B(X)$  the Borel $\sigma$-algebra of $X$, and by $A(X)$ the class of open subsets of $X$.  Moreover $L^1$ (resp.  $L^1_{\text{loc}}$) will be the usual Lebesgue space of $\mu$-integrable functions  over $X$ (resp.  $\mu$-integrable functions over compact subsets of $X$).

Let $G$ be  a countable group  of isometries  of $X$ which preserve the measure $\mu$,  acting properly discontinuously on $X$, i.e.
$\{g \in G : gK \cap K\neq \emptyset \}$ is finite for every compact set  $K\subseteq X$.   
 
 \begin{definition}[Fundamental Domain]  
A fundamental domain of $X$ for the action of $G$ is   a set  which contains almost all  representatives for the orbits of $G$  and such  that the points whose orbit has more than one representative has measure zero, i.e. a measurable set $D\subseteq X$ such that   $\mu(gD \cap  D)=0$ for every $g\in G$ with $g\neq id$, and $\mu(X\setminus G  D)=0 $.  \\
We denote by $\mathcal D$ the set of all fundamental domains of $X$. 
\end{definition} 

\begin{lemma}\label{lemfd} 
Let $E\subset X$ be a measurable set such that   $\mu(g E \cap  \tilde g E)=0$ for every $g, \tilde g\in G$ with $g\neq \tilde g$,
and let $D$ be a fundamental domain for the action of $G$. Then $\mu(E)\le \mu(D)$. 
\end{lemma}

\begin{proof}
We define $E_g=E\cap g D$ for $g\in G$. Then $\mu(E\setminus \cup_g E_g) = 0$,
$\cup_g  g^{-1} E_g\subseteq D$ and $\mu(g^{-1} E_g \cap \tilde g^{-1} E_{\tilde g})=0$ for $g\neq \tilde g$.
Therefore we have
\[
\mu(E)=\mu(\cup_g g^{-1}E_g)=\sum_g \mu(g^{-1} E_g)\le \mu(D). 
\]
\end{proof} 

\begin{corollary}\label{volumecoro} 
If $D_1, D_2$ are fundamental domains, then $\mu(D_1)=\mu(D_2)$. 
\end{corollary}

\begin{corollary}
Let $E,\,D$ be as in Lemma \ref{lemfd}. If $\mu(E)=\mu(D)$ then $E$ is also a fundamental domain. 
\end{corollary}

\begin{proof}
We have $\mu(E)=\mu(\cup_g g^{-1}E_g)=\mu(D)$, so that $\mu(D\setminus (\cup_g g^{-1}E_g))=0$. It follows that
$$
0 = \mu(X\setminus G(\cup_g g^{-1}E_g))=\mu(X\setminus G(\cup_g E_g))=\mu(X\setminus G E).
$$
\end{proof} 

\begin{assumption}\label{ass1}  \upshape 
We shall assume that there exists a  fundamental domain  $D\subseteq X$ such that $\overline D$ is compact
and $\mu(\partial D)=0$.   
\end{assumption} 

 
It is possible to show (see \cite[Lemma A1]{npst}) that Assumption \ref{ass1}  is satisfied if $X$ is compactly generated, that is, there exists a compact set $K$ such that $GK=X$, or equivalently the quotient $X/G$ is compact, and if the set of fixed points for $G$ (that is, the set of points $x\in X$ such that there exists $g\neq id$, $g\in G$, for which $gx=x$) has measure 0.  If the action of $G$ is free, that is there are no fixed points, then it is possible to show that there exists a fundamental domain with $\mu(\partial D)=0$ and $g_i D \cap  g_j D=\emptyset$ for all $g_i\neq g_j$. 
 
%
 \smallskip
 \paragraph{\bf Universal covering of  Riemannian manifolds.}  
Let  $M$ be a closed Riemannian manifold, with $\mu$ the associated volume measure,  and  let $X$ be the universal covering of $M$, with  $\pi:X\to M$ the projection map. Then $X$ is a metric measure space with the Riemannian distance and  the $\sigma$-finite Radon measure (which we still denote by $\mu$) inherited from $M$.   
  
The projection $\pi$ is a local isometry and, 
for every $p\in M$, there is a connected  neighborhood $U$ such that $\pi^{-1}(U)=\cup_i V_i$, where $V_i\cap V_j=\emptyset$ and each $V_i$   is mapped homeomorphically onto $U$ by $\pi$.  
We consider the fundamental group $\pi_1(M)$ of $M$: this identifies the group $G$ of deck trasformations (homeomorphisms of $X$ commuting with $\pi$) and we have that $M\sim X/G$.  $G$ is a countable  group and acts  properly discontinuously on $X $, and every element in the group is an isometry which preserves the measure.
 
 In this case, every fundamental domain of $X$ is a measurable set $D\subseteq X$ such that $\mu(D)=\mu(\pi(D))=\mu(M)$, and  $\mu(g_i D \cap  g_j D)=0$ for every $g_i, g_j\in G$ with $g_i\neq g_j$.   If $D$ is a fundamental domain   which is homeomorphic to an open ball, we will say that $\pi(\partial D)$ is a  spine of $M$ since   $M\setminus \pi(\partial D)$ is homeomorphic to an open ball.  
 
 A simple example is  given by the $n$-dimensional  torus $M=\T^n=\R^n/\Z^n$, where $\pi$ is the standard projection map $\R^n\to \R^n/\Z^n$ and the group $G$ is the group of discrete translations.   In this case, every fundamental domain $D$ of $M$ gives rise to a  $\Z^n$-periodic partition of $\R^n$ with cells of volume $1$, that is, $\R^n=\cup_{z\in \Z^n} (D+z)$, $|D|=1$  and $|(D+z)\cap (D+k)|=  0$ for every $z,k\in \Z^n$ with $z\neq k$.

\subsection{Perimeters} 
 
Following \cite{npst}, we define a perimeter on $X$ as a functional \[\Per: B(X) \times A(X)\to [0, +\infty]\] satisfying the following properties:
\begin{enumerate}
\item Semicontinuity: $\Per (D, U ) \leq \liminf_k \Per(D_k, U)$, if $D_k\to D$ in $L^1_{loc}$.
\item Monotonicity: $\Per(B, U)\leq \Per(B,V)$ is $U\subseteq V$.
\item Continuity: $\Per(B, U_k)\to \Per(B, U)$ if $U_k\to U $ in $L^1$ and $U_k\subset U_{k+1}$.
\item Invariance by the action of $G$: $\Per(gB, gU)=\Per(B,U)$ for all $g\in G$.
\item Compactness:  if $E_k\subseteq X$ satisfy $\sup_k\Per(E_k, U)<+\infty$ for some precompact set $U$, then up to subsequences $E_k\cap U\to E\cap U$ in $L^1$.
\item Submodularity:  let $E_1, E_2\in B(X)$, then $\Per(E_1\cap E_2)+\Per (E_1\cup E_2)\leq \Per(E_1)+\Per(E_2)$. In particular  if $E_i$ are such that $\mu(E_i\cap E_j)=0$, for $i\neq j$, then 
$\Per(\cup_i E_i)\leq   \sum_i\Per(E_i).$
\item Almost subadditivity: 
\begin{enumerate}\item
 there exists a function $\phi :\R\to  [0, +\infty)$ with $\phi(t)\to 0$ as $t\to +\infty$ 
 such that, if $U_i\cap U_j=\emptyset$ for $i\ne j$, then
 \\  $\  \Per (E, \cup_i U_i)\leq \sum_i \Per(E, U_i) \leq \Per (E, \cup_i U_i)+ \sum_{i} \mu(E\cap U_i)  \min_{j\neq i}\phi(d(U_i, U_j));$
\item there exists $c\geq 1$ such that, if $U_i\cap U_j=\emptyset$ for $i\ne j$ and $\mu(X\setminus \cup_i  U_i)=0$,  then \\ 
$\ \sum_i \Per(E, U_i)\leq c\,\Per (E).$
\end{enumerate} 
\item Local isoperimetric inequality:  there exist  $\eps>0$  and a nondecreasing function $f:[0, +\infty)\to [0, +\infty)$  with $f(0)=0$, $f'(0)=+\infty$ such that 
\[  \Per (E, B^\circ)\geq f(|E\cap B^\circ|)\qquad \forall E\subseteq \tilde M,\ \mu(E\cap   B^\circ)\leq \eps, \]  
where $B^\circ$ is the interior of a fundamental domain of $X$ which satisfies Assumption \ref{ass1}. 
\end{enumerate} 
 
Let us provide two relevant examples of perimeter functionals.

\subsubsection{\it Local perimeters.} \label{exper}
 
 
 \begin{enumerate}
 \item Let $M$ be a closed Riemannian manifold, with $\mu$ the associated volume measure,  and  $X$ the universal covering of $M$, with  $\pi:X\to M$ the projection map and $G$ the group of deck transformations.  We define a continuous function $\phi:TM\to [0, +\infty)$ which is convex, positively $1$-homogeneous, symmetric and coercive in the second variable. 
So the lifting of $\phi$ to $ TX$ defines a $G$-periodic family of norms, and the anisotropic perimeter associated to $\phi$ satisfies our assumptions. \\
As an example we can take $X=\R^n$, and $G=\Z^n$
acting by translations, so that $M=\R^n/\Z^n$ is the $n$-dimensional flat torus. 
In this case $\Per$ is the relative perimeter associated to the given family of norms, that is, $\Per(E)=\int_{\partial E} \phi(x, \nu(x))dH^{n-1}(x).$
 Another example is the hyperbolic plane $X=H$, with  its canonical volume measure $\mu$, $G$ any countable Fuchsian group (i.e. a discrete subgroup of isometries of $H$) acting properly discontinuously and cocompactly on $H$,  and $\Per$ the classical Riemannian perimeter in $H$.
\item Let $X=H$ be the Heisenberg group of topological dimension $3$, with $\mu$  its Haar measure, and $G$ the discrete Heisenberg
group generated by the triangular matrices $\left(\begin{array}{ccc} 1 & x&z\\ 0&1& y\\0&0&1\end{array}\right)$ for $x,y,z\in\Z$. Let  $X_1, X_2$    be left-invariant vector fields satisfying the H\"ormander condition, and $\Per$
be the sub-Riemannian relative perimeter functional corresponding to the choice of
$X_i$ , and $\Per$ be the sub-Riemannian perimeter functional (see for instance \cite{fssc}). 
\end{enumerate} 
  
Notice that the perimeters above satisfy the almost subadditivity property in a strict sense, i.e. with $\lambda=0$ and $c=1$
(we refer to \cite[Section 6]{npst} for more details).  
 
\subsubsection{\it Nonlocal perimeters.} \label{experfr}
Let $X=\R^n$, $\mu$ the Lebesgue measure, $G=\Z^n$ 
acting by translations,  and $K:\R^n\to \R$ an interaction kernel satisfying
\begin{itemize}\item $K(h)= K(-h)$ for all $h\in \R^n$,\item $\min(|h|,1) K(h)\in L^1(\R^n)$, \item there exists $C>0$ and $s\in (0,1)$ such that $K(h)\geq C |h|^{-n-s}$.  
\end{itemize} 
For $E\subseteq \R^n$ we define the nonlocal perimeter of $E$ as follows:  
 \begin{equation}\label{frac}
  \Per(E):=  \int_{E }\int_{\R^n\setminus E}    K( x-y)dxdy  \end{equation}
 and its localized version, for  $U\subseteq \R^n$ open set, as 
 \[
 \Per(E, U):=  \int_{E\cap U}\int_{\R^n\setminus E}   K( x-y)dxdy+ \int_{E\setminus U}\int_{U\setminus E}    K( x-y) dxdy
 \]\[= \frac{1}{2}\int_{U\times U} (\chi_E(x)-\chi_E(y))^2      K( x-y)d xdy+ \int_{U\times \R^n\setminus U} (\chi_E(x)-\chi_E(y))^2    K( x-y) dxdy.\] 
 The lower semicontinuity with respect to $L^1_{\text{loc}}$ convergence   is a direct consequence of Fatou's Lemma.  The monotonicity and the continuity property with respect to increasing sequences of open sets $U_i$ are consequences of the definition of the perimeter and of the monotone convergence theorem. The invariance with respect to the action of the group $G$ (i.e., the invariance with respect to translations) follows from the definition. 
 
 The compactness property has been proved in \cite[Theorem 1.2]{jw} (see also \cite{bs}), in particular the assumption $K\not\in L^1(B(0,1))$ is necessary to get this result. As for the subadditivity it is easy to check (see e.g. \cite[Lemma 2.4]{dnrv}) that
 if $U_i\cap U_j=\emptyset$ for $i\ne j$, then
 \begin{eqnarray*} && \Per(E, \cup_i U_i)- \sum_i \Per(E, U_i) = \\ 
 &=& - \sum_{i\ne j}\int_{E\cap U_i}\int_{U_j\setminus E}  K(x-y)dxdy   
 \geq  - \sum_i \sum_{j:\,i\ne j} \int_{E\cap U_i}\int_{U_j}  K(x-y)dxdy \\
 &\geq& - \sum_i |E\cap U_i| \int_{\R^n\setminus B_{d_i}(0)} K(h)dh
 =  - \sum_i |E\cap U_i| \phi(d_i), 
  \end{eqnarray*}
 where we let $d_i:= \min_{j\ne i} dist(U_i,U_j)$ and
 \[
 \phi(t) := \int_{\R^n\setminus B_{t}(0)} K(h)dh.
 \]
 Moreover if $\cup_i U_i=\R^n$ and $U_i\cap U_j=\emptyset$,  then
\begin{eqnarray*} && \sum_i\Per(E, U_i) =\sum_i \int_{E\cap U_i}\int_{\R^n\setminus E} K(x-y)dxdy+ \int_{E\setminus U_i}\int_{U_i\setminus E} K(x-y)dxdy \\ 
 &\leq & \sum_i \int_{E\cap U_i}\int_{\R^n\setminus E} K(x-y)dxdy+ \sum_i \int_{E}\int_{U_i\setminus E} K(x-y)dxdy\\ &=&
  \int_{E\cap \cup_i U_i}\int_{\R^n\setminus E} K(x-y)dxdy+ \int_{E}\int_{\cup_i U_i\setminus E} K(x-y)dxdy =2\Per (E).  \end{eqnarray*}
  On the other hand, the submodularity is easily checked since 
  \[ \Per(E_1\cup E_2)+\Per(E_1\cap E_2)=\Per(E_1)+\Per(E_2)-2\int_{E_2\setminus E_1}\int_{E_1\setminus E_2} K(x-y)dxdy.\]
  By applying recursively this formula in the case of a family $E_i$, with $E_i\cap E_j=\emptyset$, we get
 \begin{equation}\label{unione} \Per(\cup_i E_i)=\sum_i \Per(E_i) -2\sum_{i\neq j} \int_{E_i}\int_{E_j} K(x-y)dxdy.\end{equation}

  
  In the special case $K(h)=|h|^{-n-s}$ for $s\in (0,1)$, the nonlocal perimeter is usually called fractional perimeter and has been 
  introduced and studied in \cite{crs}  (see also \cites{cn,cm}).
  As for the isoperimetric inequality, in \cite[Lemma 2.5]{dnrv} it is proved that if $U$ is a bounded open set, then for every $E$ with $|E\cap U|\leq |U|/2$, there holds: \[\Per(E,U)\geq C |E\cap U|^{\frac{n-s}{n}}.\]
 For general kernels satisfying the  assumption that $K(h)\geq C|h|^{-n-s}$, the same inequality easily follows.

\subsection{Partitions}

\begin{definition}
A partition of $X$ is a collection of measurable subsets $\{E_k\}_{k\in\mathbb I}$,  where $\mathbb I$ is either a finite or a countable set of ordered indices, such that 
\begin{enumerate} 
\item $\mu(E_k)>0$ for all $k$,
\item $\mu(E_k\cap E_j)=0$ for all $k\ne j$,
\item $\mu(X\setminus\cup_k E_k)=0$.
\end{enumerate}
\end{definition}

 We introduce the notion of topological boundary of a partition  $\{E_k\}_{k\in\mathbb I}$ as
\[\partial \{E_k\}_{k\in\mathbb I}:=
\{x\in \R^n\ |\ \text{ for every $\rho>0$ there exists $k\in \mathbb I$ s.t.}\ 
0<\mu(E_k\cap B(x, \rho))<\mu(B(x, \rho))\}\] 
and the notion of reduced boundary as 
\[\partial^*\{E_k\}_{k\in\mathbb I}:=\bigcup_{i\in \mathbb I} \bigcup_{j<i} \partial^* E_i\cap \partial^* E_j,\]  
where $\partial^* E_i$ is the reduced boundary of $E_i$ (see \cite{maggi}).
 
We also recall the definition of conical partitions and blow-up of a partition. 
\begin{definition}\label{conical} 
A partition $\{E_k\}_{k\in\mathbb I}$ is conical (with vertex $0$) if  $rE_k=E_k$ for every $r>0$ and every $k\in \mathbb I$. \end{definition} 

\begin{definition}\label{blowup} The blow-up of a partition  $\{E_k\}_{k\in\mathbb I}$ at $x\in\R^n$ and scale $\eps>0$ is the partition defined as 
\[E^{x,\eps}_k:=\frac{E_k-x}{\eps}\qquad k\in\mathbb I. \]  
The regular set of the partition  is the set of points $x\in \partial \{E_k\}_k$  such that there exist an open half-space
$H\subseteq \R^n$ and $i,j\in \mathbb{I}$ such that as $\eps\to 0$, $E_i^{x,\eps}\to H$, $E^{x,\eps}_j\to \R^n\setminus H$ and $E^{x,\eps}_k\to \emptyset$ for $k\neq i,j$, in $L^1_{\text{loc}}(\R^n)$.  
\end{definition}

\noindent Notice that a fundamental domain $D$ naturally induces the $G$-periodic partition $\{E_g\}_{g\in G}$, where $E_g = g D$.

\begin{proposition} \label{locallyfinite} 
If the fundamental domain $D$ is precompact, then the  $G$-periodic partition $\{E_g\}_{g\in G}$ induced by $D$ is locally finite. 
\end{proposition} 

\begin{proof} Let us consider a compact set $K$. Note that $K=\cup_{g\in H} gD\cap K$, where $H=\{g\in G:\ gD\cap K\neq \emptyset\}$. We want to prove that $H$ is finite. Indeed, let us consider the compact set $\tilde K:=\overline D\cup K$.
Since the the group $G$ acts properly discontinuously, we have
 $g\tilde K\cap \tilde K\neq\emptyset$ only for $g$ in a finite subset $\tilde G$ of $G$. 
 It is immediate to observe that $H\subseteq \tilde G$. 
\end{proof} 

\begin{definition}\label{defminimi}
We say that a partition $\{E_k\}_k$ is $\Lambda$-minimal in an open set $A\subset X$, for some $\Lambda\ge 0$, if
$\sum_k \Per(E_k,A)<+\infty$ and 
$$
\sum_k \Per(E_k,A)\le \sum_k \big[\Per(F_k,A) + \Lambda \mu(E_k\Delta F_k)\big],
$$ 
for every partition $\{F_k\}_k$ of $X$ such that $E_k\Delta F_k \Subset A$ for all $k$.\\
We say that the partition is $(\Lambda,r)$-minimal (see \cite{maggi}) for some $\Lambda\ge 0$ and $r>0$, 
if  it is $\Lambda $-minimal in $B_r(x)$ for all $x\in X$.\\
\end{definition}

We observe that, for conical partitions, being $(\Lambda,r)$-minimal  is equivalent to being $(0, \infty)$-minimal. 

 \section{Existence of minimal $G$-periodic partitions} \label{secex}
 
 \begin{definition}   
 A fundamental domain is called isoperimetric if it is a solution to the minimization problem 
 \begin{equation}\label{iso} \inf\{ \Per(D)\ | \ D\in \mathcal{D}\}.
 \end{equation} 
 \end{definition} 
 
In this section we prove existence of solutions to the isoperimetric problem \eqref{iso}.  
Notice also that this is equivalent to find a minimal partition among all possible $G$-periodic partitions.

 \begin{theorem}\label{isothm}
 There exists an isoperimetric fundamental domain $D$.
 \end{theorem} 
The proof of this theorem is based on  two basic tools:  the lower semicontinuity of the perimeter and a concentration compactness lemma. 
These results are  generalizations of the analogous results in \cite{npst}. 

 For a sequence $g_k\in G$ we will write that $\lim_k g_k=+\infty$ if for every finite subset $F\subseteq G$, the set $\{g_k\in F\}$ is finite.  In particular, since $G$ acts properly discontinuously on $X$, we have that for any $K\subseteq X$ compact and for every $N>0$ there exists $k_N$ such that $d(g_k K, K)\ge N$ for all $k\geq k_N$. 

\begin{lemma}[Semicontinuity] \label{lsctheorem} Assume that $E_k\subseteq X$ is a sequence of measurable sets, $g_k^i\in G$ such that $(g_k^i)^{-1}\circ g_k^j\to +\infty$ if $j\neq i$ as $k\to +\infty$ and $(g_k^{i})^{-1}E_k\to E^i$ in $L^1_{\rm loc}$ as $k\to +\infty$. 
Then \[\sum_i \Per (E^i)\leq \liminf_k \Per(E_k). \]
\end{lemma}

\begin{proof}
The proof follows along the same lines as in  \cite[Theorem 3]{npst}, the main difference being the use of the almost  subadditivity instead of subadditivity.  Let $x\in X$ and let $U=B_r(x)$, so by the assumption and the fact that $G$ acts properly discontinuously on $X$ we get that  for $k$ sufficiently large $g^i_k U\cap g^j_k U=\emptyset$ for $i\neq j$, and $dist(g_k^i U, g_k^j U)\to +\infty$ as $k\to +\infty$, uniformly in $i\ne j$.  
Using the property of the perimeter functional, we get
\begin{eqnarray*} 
\liminf_k \Per(E_k)&\geq&  \liminf_k \Per(E_k, \cup_i g_k^i U) \\ & \geq &
 \liminf_k \left[\sum_i \Per(E_k, g_k^i U) - \sum_{i} \mu(E_k\cap g_k^i U)  \min_{j\neq i}\phi(dist(g_k^i U, g_k^j U))\right]
\\&\geq & \sum_{i} \liminf_k \Per(E_k, g_k^i U)- \sum_{i} \limsup_k \mu(E_k\cap g_k^i U)  \min_{j\neq i}\phi(dist(g_k^i U, g_k^j U))\\
&=& \sum_i \liminf_k \Per ((g_k^1)^{-1}E_k, U)\geq \sum_i \Per(E^i, U),  
\end{eqnarray*} 
from which we obtain the thesis by sending $r\to +\infty$. 
\end{proof} 

\begin{lemma}[Concentration compactness]  \label{cctheorem} 
Assume that $E_k\subseteq X$ is a sequence of measurable sets, with $|E_k|=m$ and $\sup_k \Per(E_k)\leq C<+\infty$.

Then there exist a subsequence $E_k$, $g_k^i\in G$  for $i\in I\subseteq \N$ and $E^i\subseteq X$ measurable sets  such that 
 $(g_k^i)^{-1}\circ g_k^j\to +\infty$ if $j\neq i$ as $k\to +\infty$,  $(g_k^{i})^{-1}E_k\to E^i$ locally in $L^1$ as $k\to +\infty$  and $\sum_i |E^i|=m$. 
\end{lemma}

\begin{proof} 
The proof is an adaptation of the proof of Theorem 3.3 in \cite{npst}. 

First of all we consider $h_k^i$ an enumeration of $G$ such that $i\mapsto |E_k\cap h_k^i B^\circ|$ is nonincreasing, where $B^\circ$ is  the interior of a fundamental domain which satisfies Assumption \ref{ass1} (see condition (7) in the definition of the perimeter functional).  By the compactness property of the perimeter we have that up to a subsequence $(h_k^{i})^{-1}E_k
\cap B^\circ \to F^i$ in $L^1 $ as $k\to +\infty$. 

Moreover, since $\mu(h_k^iB^\circ\cap h_k^j B^\circ)=0$ if $i\neq j$, we have that 
\[|E_k\cap h_k^n B^\circ|\leq \frac{1}{n} \sum_{i=1}^n  |E_k\cap h_k^i B^\circ|\leq \frac{1}{n}|E_k|=\frac{m}{n}
\qquad \text{and} \qquad \sum_{n=1}^{+\infty}  |E_k\cap h_k^n B^\circ|= |E_k|=m. \]

Recalling condition (7), for every $\delta>0$ there exists $\eta_0$ such that $r\leq \delta f(r) $ for all $r\in [0, \eta_0]$. Without loss of generality 
 we may  choose $\eta_0<\eps$, where $\eps$ is again as in assumption 7. So, for every $\delta>0$, 
we have that there exists $\bar n=\bar n(\eta_0)$  for which  
 $\mu(E_k\cap h_k^n B^\circ)\leq \delta  f(\mu(E_k\cap h_k^n B^\circ))$, and  $\mu(E_k\cap h_k^n B^\circ)\leq  \eta_0\leq \eps$ for all $n\geq \bar n$ and for all $k$.
In particular by the local isoperimetric inequality  and the almost subadditivity property we have that   
\begin{eqnarray*} && \sum_{n=\bar n}^{+\infty} \mu(E_k\cap h_k^n B^\circ) =\sum_{n=\bar n}^{+\infty} \frac{\mu(E_k\cap h_k^n B^\circ)}{ f(\mu(E_k\cap h_k^n B^\circ)} f(\mu(E_k\cap h_k^n B^\circ)) \\ & \leq& \delta \sum_{n=\bar n}^{+\infty} f(\mu(E_k\cap h_k^n B^\circ))\leq \delta \sum_{n=\bar n}^{+\infty} \Per((h_k^{n})^{-1}E_k, B^\circ) \\ & \leq &  \delta \sum_{n=1}^{+\infty} \    \Per(E_k, h_k^n B^\circ)  
 \leq \delta c \Per(E_k)\leq \delta c C .
\end{eqnarray*} 
Since $\delta>0$ was arbitrary, we get that  \[\lim_{ n\to +\infty} \sup_k \sum_{i=n}^{+\infty}  \mu(E_k\cap h_k^i B^\circ)=0.\]
On the other hand, by construction we have that
\[ \lim_{k\to +\infty}  \sum_{n=1}^{+\infty}   \mu(E_k\cap h_k^n B^\circ)=m. \]
This is sufficient to conclude that   $\sum_{n=1}^{+\infty} |F^n|=m$. Indeed, given $\eps>0$ we have that for  $n\geq n(\eps)$   and for every $k$ there holds 
\[  m-\eps\leq \sum_{i=1}^{n}   \mu(E_k\cap h_k^n B^\circ)=\sum_{i=1}^{n}   \mu((h_k^{n})^{-1}E_k\cap  B^\circ)\leq m,\] 
and we conclude by sending $k\to +\infty$ and then $n\to +\infty$ (see also \cite[Lemma B1]{npst}). 

Now, we define an equivalence relation $j\sim i$ if the set $\{(h_k^i)^{-1}\circ h_k^j\}$ is finite. Let $I$ be the quotient set, let $[i]\in I$ an equivalence class and denote by $\underline i=\min \{i\in [i]\}$. We define $g_k^{i}:= h_k^{\underline i}$.  Up to passing to a subsequence in $k$, we may assume that for all $i\in [i]$,  $(h_k^{\underline i})^{-1}\circ h_k^{i}$  is constant $=h^i$.
So, $h_k^{i'} =g_k^{i}\circ h^{i'}$ and by construction $(g_k^i)^{-1}g_k^j\to +\infty$ as $k\to +\infty$. 
We have that, up to a subsequence, $(g_k^{i})^{-1}E_k \to E^i$ in $L^1_{\text{loc}} $ as $k\to +\infty$.  
  
By construction, if $i'\in [i]$ we have that 
\[
\mu(F^{i'})=\lim_k\mu((h_k^{i'})^{-1}E_k \cap B^\circ)= \mu(E^i\cap (h^{i'})^{-1} B^\circ).
\] 
Recalling that $B^0$ is a fundamental domain, 
  \[ \mu(E^i)\geq \sum_{i'\in [i]}   \mu(E^i\cap (h^{i'})^{-1} B^\circ) = \sum_{i'\in [i]}  \mu (F^{i'}).\]
  This implies that $\sum_{i\in I} \mu(E^i)=m$. 
  \end{proof}

 \begin{proof}[Proof or Theorem \ref{isothm}] If $\Per(D)=+\infty$ for every fundamental domain $D$ in $X$, 
 then there is nothing to prove.\\ 
 Assume that $\inf\{ \Per(D)\ | \  D\in\mathcal{D}\}<+\infty$ and let $E_k$ be a sequence of fundamental domains such that 
$$
\liminf_k \Per(E_k)=  \inf\{ \Per(D)\ | \   D\in\mathcal{D} \}.
$$ 
By  Lemma \ref{cctheorem}, for $i\in I\subseteq \N$  there exist $g_k^i\in G$ and $E^i\subseteq X$  such that 
 $(g_k^i)^{-}\circ g_k^j\to +\infty$ if $j\neq i$ as $k\to +\infty$,  $(g_k^{i})^{-1}E_k\to E^i$ locally in $L^1$ as $k\to +\infty$  and $\sum_i |E^i|=m$. 
 Moreover, by Lemma \ref{lsctheorem}  we have that
 \[\sum_i \Per (E^i)\leq \liminf_k \Per(E_k)=  \inf\{ \Per(D)\ | \ D\in \mathcal D\}.  \]
 
 We claim that $\mu(E^i\cap E^j)=0$ for all $i\neq j$. Assume by contradiction it is not the case. Then $\mu(E^i\cap E^j)>0$. 
  By $L^1$ convergence we have that $\mu ((g_k^{i})^{-1}E_k\cap (g_k^{j})^{-1}E_k) >0$  for every $k$ sufficiently large, which is in contradiction with the fact that $E_k$ is a fundamental domain. 
 In the same way we may show that $\mu(E^i\cap gE^j)=0$ for all $g\in G$, and then $\mu(g(\cup_i E^i) \cap (\cup_i E^i))=0$.
 
Finally, the fact that $\mu(\cup_i E^i)=m=\mu (E_k)$ and the fact that  $\mu(g(\cup_i E^i) \cap (\cup_i E^i))=0$ implies that $\mu(X\setminus G( \cup_i E^i ))=0$, so that $\widetilde D:=\cup_i E^i$ is a fundamental domain. By the submodularity property, we then conclude that
\[
\Per(\widetilde D)\leq \sum_i \Per (E^i)\leq \liminf_k \Per(E_k)=  \inf\{ \Per(D)\ | \  D\in \mathcal D\}, 
\]
which gives that $ \Per(\widetilde D)=\sum_i \Per (E^i)=  \inf\{ \Per(D)\ | \  D\in \mathcal D\}$. 
\end{proof}  

\begin{remark}\upshape 
Note that, by the proof of  Theorem   \ref{isothm}, if $D$ is an isoperimetric fundamental domain with $\Per(D)<+\infty$, then $D=\cup_i E_i$ 
with $\mu(E_i\cap E_j)=0$ and $\Per(D)= \sum_i \Per (E^i)$. Therefore, if $\Per$ is a nonlocal perimeter as in \eqref{frac} we get by \eqref{unione}  that necessarily $D=E^i$ for some index $i$, and $\mu(E^j)=0$ for all $j\neq i$.
\end{remark}

\section{Minimal partitions in $\R^n$}\label{secloc}

We now consider the particular case $X=\R^n$, equipped with the Lebsgue measure and with the Euclidean distance, and we fix $G= \Z^n$.

As above, every fundamental domain $D$  generates a  $G$-periodic partition of $\R^n$ into regions  $gD$ for $g\in G$, which have all the same volume. 
We observe that, if $r<\frac12$, then $g B_r(x)\cap B_r(x) =\emptyset$ for all $g\in G\setminus\{0\}$ and for all $x\in\R^n$. 

We shall also assume that $\Per$ is either the local anisotropic perimeter, induced by a $G$-periodic family of norms $\phi$ on $\R^n\times \R^n$ (see Paragraph \ref{exper}) or  the nonlocal fractional  perimeter (see Paragraph  \ref{experfr}). 
  
In the case $\Per$  is the anisotropic perimeter, all the results  can be easily extended to the case when $X$ is the universal covering of a closed Riemannian manifold $M$ (see Example 1 in Paragraph \ref{exper}) and $r$ is smaller than the injectivity radius of $M$, with respect to the distance induced by $\p$.
  
\begin{proposition}\label{partmin}
Let $D$ be an isoperimetric fundamental domain. Then $D$ generates a $G$-periodic partition of $\R^n$ which is $(\Lambda,r)$-minimal for every $r<\frac12$,
with $\Lambda = 0$ in the case of the local perimeter, and $\Lambda=\int_{\R^n\setminus B_{1-2r}(0)} K(h) dh,$ in the case of the nonlocal perimeter.
\end{proposition} 

\begin{proof}
Let us fix a ball $B$ of radius $r<1/2$, and let $\{E_g\}_{g\in G}$ be a partition of $\R^n$ such that $E_g\Delta gD\Subset B$ for all $g\in G$.
We now define 
\[
\widehat D:= \left( D\setminus  \bigcup_{g\in G} (g^{-1}B)\right)\cup \bigcup_{g\in G} (g^{-1}E_g).
\]
Note that  $\widehat D$ is also a fundamental domain and so from  the minimality of $D$ we get
\begin{eqnarray*}
0&\le& \Per(\widehat D)-\Per(D) = \Per(\widehat D, \cup_{g\in G} (g^{-1}B))-\Per(D, \cup_{g\in G} (g^{-1}B))
\\
&\le& \sum_{g\in G} \left[\Per(\widehat D, g^{-1}B)-\Per(D, g^{-1}B)\right]
+ \phi(1-2r) \sum_{g\in G} |( D\Delta \widehat D)\cap g^{-1}B|
\\
&=& \sum_{g\in G} \Per(E_g,B)-\Per(gD,B) + \phi(1-2r) \sum_{g\in G}|( gD\Delta E_g)\cap B|,
\end{eqnarray*}
where $\phi(t)\equiv 0$ if $\Per$ is an anisotropic perimeter, and 
\[
\phi(t)=\int_{\R^n\setminus B_t(0)} K(h) dh,
\]
if $\Per$ is a nonlocal perimeter.
\end{proof}

\subsection{Elimination Lemma and conical partitions} \ \ \\
We provide some  preliminary results  in order to get the local finiteness of $(\Lambda, r)$-minimal partitions. 
First of all we get an a priori estimate on the  perimeter of a $(\Lambda, r)$-minimal partition, then we state an  Elimination Lemma,  and we apply it, with a dimension reduction argument, to show that conical partitions are finite. 
 
 In all these results we will assume that $\Per$ is either the local anisotropic perimeter, 
 or the fractional $s$-perimeter induced by the kernel 
  \[
  K_s(h) = \frac{1}{|h|^{n+s}} \quad s\in (0,1).
  \] 
  
 \begin{lemma} \label{estimate} 
  Assume that $\Per$ is either the local anisotropic perimeter or the fractional perimeter.
 Let $(E_i)_{i\in \N}$ be a $(\Lambda,r)$-minimal partition for  some $r<\frac12$.

Then, there exists $C_0>0$ depending on $n,s,\Lambda$ such that 
\[\sum_i \Per(E_i\cap B_r(x))\leq C_0 r^{k}\qquad\text{and }\qquad \sum_i \Per(E_i, B_r(x))\leq C_0 r^{k}, 
\] 
where $k=n-1$ in the case of the local perimeter, and $k=n-s$ in the case of the fractional perimeter.
 \end{lemma} 
 
 \begin{proof} For the local case we refer to \cite[Theorem 1]{mt}  and for the fractional case to \cite[Corollary 3.6]{cm}. 
 We sketch the proof in the latter case. 
Let $\{F_i\}_{i\in \mathbb N}$ be the partition defined as 
\[
F_i:=\left\{
\begin{array}{ll}
E_1\cup B_r(x)& \text{if $i=1$}
\\
E_i\setminus B_r(x)& \text{if $i\neq 1 $.}
\end{array}
\right.\]
so that $E_i\Delta F_i\subset B_r(x)$ for all $i$. Then,
by Proposition \ref{partmin} we get that
\begin{eqnarray}\label{stima1}
& \Lambda \sum_{i\neq 1} |E_i\cap B_r(x)|+\Lambda |B_r(x)|=2\Lambda\omega_n r^n \ge  \sum_i \Per(E_i,B_r(x)) - \Per(F_i,B_r(x)) 
\\
= &\nonumber
\Per(E_1,B_r(x)) - \Per(E_1\cup B_r(x), B_r(x)) + \sum_{i>1}  \Per(E_i,B_r(x)) - \Per(E_i\setminus B_r(x) ,B_r(x)).
\end{eqnarray}

Let us denote $J(A,B)=\int_A\int_B K(x-y)dxdy$ and $B_r=B_r(x)$, and  we observe that 
\begin{eqnarray*}\Per(E_1,B_r) & =&\Per(E_1\cap B_r) -J(E_1\cap B_r, E_1\setminus B_r)+J( B_r \setminus E_1, E_1\setminus B_r ) \\ 
  \Per(E_1\cup  B_r  ,B_r)&=& J(B_r, \R^n\setminus (E_1\cup B_r)=\Per(B_r)-J( B_r, E_1\setminus B_r)\leq \Per(B_r) \\ \Per(E_1,B_r) - \Per(E_1\cup B_r , B_r) &\geq &\Per(E_1,B_r)-\Per(B_r)\\ 
\Per(E_1,B_r) - \Per(E_1\cup B_r , B_r) &=&
\Per(E_1\cap B_r)- \Per(B_r)+2J( B_r\setminus E_1, E_1\setminus B_r)\\ &\geq  &
\Per(E_1\cap B_r)- \Per(B_r)
\end{eqnarray*} 
and that
\begin{eqnarray*}\Per(E_i,B_r) & =&\Per(E_i\cap B_r)) -J(E_i\cap B_r, E_i\setminus B_r)+J( B_r\setminus E_i, E_i\setminus B_r) \\ 
  \Per(E_i\setminus  B_r ,B_r)&=& J(B_r, E_i\setminus B_r)  \\\sum_{i>1} \Per(E_i,B_r) - \Per(E_i\setminus B_r, B_r) &=&\sum_{i>1} 
  \Per(E_i,B_r)- J(B_r, \cup_i E_i\setminus B_r)\\ &\geq& \sum_{i>1} \Per(E_i,B_r)-\Per(B_r) \\ 
\sum_{i>1} \Per(E_i,B_r) - \Per(E_i\setminus B_r, B_r) &=&
\sum_{i>1}\Per(E_i\cap B_r)-2J(E_i\cap B_r, E_i\setminus B_r)\\  &\geq & \sum_{i>1}  \Per(E_i\cap B_r)-2J( B_r, \cup_i E_i\setminus B_r)\\ 
&\geq & \sum_{i>1}  \Per(E_i\cap B_r)-2\Per(B_r).
\end{eqnarray*} 
 Substituting in \eqref{stima1} and recalling that $\Per(B_r)= c_0 r^{n-s}$  for a constant $c_0$ depending on $n,s$, we get
 \[2\Lambda\omega_n r^n \ge \sum_i\Per( E_i\cap B_r)-3c_0 r^{n-s}\qquad  2\Lambda\omega_n r^n \ge \sum_i\Per( E_i, B_r)-2c_0 r^{n-s}\]
 from which we obtain the thesis. 
 \end{proof} 

We recall  an iteration lemma, whose proof can be easily obtained  by induction.

\begin{lemma}[De Giorgi iteration lemma]\label{lemmadg} 
Let $\alpha\in (0,1)$, $M>0$, $L>1$ and $u_k$ a decreasing sequence of positive numbers such that 
\[u_{k+1}^{1-\alpha}\leq L^k M u_k\qquad \text{and}\qquad u_0\leq \frac{1}{M^{\frac{1}{\alpha}}L^{\frac{1-\alpha}{\alpha^2}}}\] then $u_k\to 0$.
\end{lemma} 

We now extend to our setting an important result which is known for locally finite partitions 
(we refer to \cite[Theorem 2.4]{ct2} and \cite[Lemma 30.2]{maggi}  for the proof in the local case, and to \cite[Lemma 3.4] {cm} for the fractional case). 

 \begin{lemma}[Elimination Lemma]\label{infiltration}
  Assume that $\Per$ is either the local anisotropic perimeter or the fractional perimeter.
  Let $(E_i)_{i\in \N}$ be a $(\Lambda,r)$-minimal partition for some $r<\frac12$.
Then for every $N>0$ there exist positive constants $\sigma_0>0$ depending on $N, n, s$, and  $r_0<\frac{1}{2}$ depending on $n$ in the local case and on $n,s, \Lambda, N$ in the fractional case such that 
\[\text{ if for $ r<r_0$ there holds  }|\cup_{i>N}E_i\cap B_r(x)| \leq \sigma_0 r^n, \quad \text{then }|\cup_{i>N}E_i\cap B_{\frac{r}{2}}(x)|=0.\] \end{lemma} 

\begin{proof}
We show the result in the case of the fractional perimeter, being the local case a straightforward adaptation of the proof. 

We fix $x\in\R^n$ and   $N\in \mathbb N$, we let $V:=\cup_{i>N}E_i$ and 
$u(r):= |V\cap B_r(x)|$ for $0<r<1/2$. We have to show that if $u(r)\leq \sigma r^n$ then $u(r/2)=0$.

For $j\in\{1,\ldots,N\}$ we also let $\{F^{j}_i\}_{i\in \mathbb N}$ be the partition defined as 
\[
F^{j}_i:=\left\{
\begin{array}{ll}
E_i & \text{if $i\le N$ and $i\ne j$,}
\\
E_j \cup (V\cap B_r(x)) & \text{if $i= j$,}
\\
E_i\setminus B_r(x) & \text{if $i>N$,}
\end{array}
\right.\]
so that $E_i\Delta F^{j}_i\subset B_r(x)$ for all $i$. Then,
by Proposition \ref{partmin} we get that
\begin{eqnarray*}
 \Lambda u(r) &\ge& \sum_i \Per(E_i,B_r(x)) - \Per(F^{j}_i,B_r(x)) 
\\
&=& 
\Per(E_j,B_r(x)) - \Per(F^{j}_j,B_r(x)) + \sum_{i>N}  \Per(E_i,B_r(x)) - \Per(E_i\setminus B_r(x) ,B_r(x)) 
\\
&\ge&
\Per(E_j,B_r(x)) - \Per(E_j\cup (V\cap B_r(x)),B_r(x)) + \Per(V,B_r(x)) - \Per(V\setminus B_r(x) ,B_r(x)),
\end{eqnarray*}
where we used the submodularity  of the perimeter and the fact that 
\[\Per(V\setminus B_r(x) ,B_r(x)) = \int_{V\setminus B_r(x)}\int_{B_r(x)}\frac{1}{|x-y|^{n+s}}dxdy=\sum_{i>N}\ \Per(E_i\setminus B_r(x) ,B_r(x)) .\]
Averaging over $j\in\{1,\ldots,N\}$ and arguing exactly as in the proof of \cite[Lemma 3.4]{cm} we then obtain that
\begin{equation}\label{eqode}
C_1 u(r)^\frac{n-s}{n} \le C_2(1+N)\int_0^r\frac{u'(t)}{(r-t)^s}\,dt + \Lambda N u(r),
\end{equation}
where $C_1,\,C_2$ are positive constants depending only on $s$ and $n$. Now we choose $r_0>0$ such that  
\[ r_0^s<\min \left(\frac{1}{2^s},  \frac{C_1}{2\Lambda N \omega_n^{\frac{s}{n}}}\right) \]
and we get  for all $r\leq r_0$, 
\[\Lambda N u(r)\leq \Lambda N  u(r)^\frac{n-s}{n} \left( \omega_n r^n\right)^{\frac{s}{n}}\leq \frac{C_1}{2} u(r)^\frac{n-s}{n}. \]
We substitute this inequality in  \eqref{eqode}  and then integrate  \eqref{eqode}  between $0$ and $l<r$, so that we get 
\begin{equation}\label{eqode2}
 \int_0^l u(r)^\frac{n-s}{n}dr  \le \frac{2C_2(1+N)}{(1-s)C_1} l^{1-s} u(l) .
\end{equation}
Let \[\sigma_0:=\left(\frac{(1-s)C_1}{8C_2(1+N)}\right)^{\frac{n}{s}}2^{-\frac{n(n-s)}{s^2}}, \]
and assume that there exists $\bar r<r_0$ such that $u(\bar r)\leq \sigma_0\bar r^n$.  Define the  sequence $r_k:=\frac{\bar r}{2}+\frac{\bar r}{2^{k+1}}$, and let $u_k:= u(r_k)$. Then, by definition, $u_0=u(\bar r)$ 
and $\lim_k u_k=u\left(\frac{\bar r}{2}\right)$. 

Now we let $l=r_k$ in \eqref{eqode2}, so that
\[ u_{k+1}^{\frac{n-s}{n}}\frac{\bar r}{2^{k+2}}\leq \int_{r_{k+1}}^{r_k} u(r)^\frac{n-s}{n}dr \leq \int_{0}^{r_k} u(r)^\frac{n-s}{n}dr \leq  \frac{2C_2(1+N)}{(1-s)C_1} r_k^{1-s} u_k\leq   \frac{2C_2(1+N)}{(1-s)C_1} \bar r^{1-s} u_k,  \]
which implies that
\[u_{k+1}^{1-\frac{s}{n}} \leq 2^{k+2} \frac{2C_2(1+N)}{(1-s)C_1} \frac{1}{\bar r^s} u_k. 
\]
We now apply Lemma \ref{lemmadg} to the sequence $u_k$, with  $\alpha=\frac{n}{s}$, $L=2$ and  $M=\frac{8C_2(1+N)}{(1-s)C_1\bar r^s}$, 
and we deduce that $u_k\to 0$ as $k\to +\infty$, which gives the thesis.
\end{proof} 

\begin{proposition}\label{conicalpart}  Assume that $\Per$ is either the local anisotropic perimeter or  the fractional perimeter. Let $\{E_i\}_{i\in\mathbb I}$ conical partition, which is $(0, \infty)$-minimal. Then $\mathbb I$ is finite. 
\end{proposition} 
\begin{proof} This result is proved for the classical perimeter in \cite[Theorem 8]{mt}. 

Observe that if $n=1$, the only conical partition is given by $(-\infty, 0), (0, +\infty)$, so it is finite. Assume that $n\geq 2$, and fix $B_1(0)$. Assume that the partition $\{E_i\}$ is countable (not finite). Eventually passing to a subsequence, we fix a sequence  $x_i\in \partial B_r\cap \partial E_i$ with $x_i\to x_0\in \partial B_1$. We now consider the blow-up  at $x_0$ at scale $r>0$ of the partition $\{E_i\}$, see Definition \ref{blowup}.  By the estimate in Lemma \ref{estimate}, and by the rescaling properties of the fractional perimeter and of the local perimeter, we get that $\Per(E_i^{x_0, \eps}, B_r)\leq C_0$ for every $r>0$. So, by the compactness property of $\Per$, we get that, up to passing to a subsequence, $E_i^{x_0, \eps}\to Q_i$ as $\eps\to 0$, locally in $L^1(\R^n)$ and in particular $E_i^{x_0, \eps}\to Q_i$ as $\eps\to 0$,  in $L^1(B_{1}(x_0)$. 

By  semicontinuity properties, see Lemma \ref{lsctheorem}, also $\{Q_i\}$ is $(0, \infty)$-minimal. 
We observe now that if the partition $\{E_i\}$ were not finite, then necessarily, also the partition $\{Q_i\}$ is not finite. Indeed if it were not the case, then by the $L^1$ convergence we could find $N>0$ elements of the partition  $E_i$, for $i=1, \dots, N$, such that for $r>0$ sufficiently small, $\cup_{i>N}E_i\cap B_r(x_0)| \leq \sigma_0 r^n$  where $\sigma_0$ is as in Elimination Lemma \ref{infiltration}. Therefore, by the Elimination Lemma \ref{infiltration}, we conclude that $Q_i=\emptyset$ for $i>N$, and so also $|E_i\cap B(x_0, r)|=0$ for $i>N$, in contradiction with the fact that the conical partition  $\{E_i\}$ is not finite. 

Since the partition $\{E_i\}$ is conical, then the partition $\{Q_i\}$ is given by cylinders with a common direction, so, up to a rotation of coordinates, we may write for all $i\in \mathbb{I}$,  $Q_i=C_i\times \R$, for some $C_i$ cone of vertex $x_0$. It is easy to check that $\{C_i-x_0\}$ is a $(0, \infty)$-minimal conical partition in $\R^{n-1}$ (for the fractional setting see \cite[Proposition 3.11]{cm} and \cite[Theorem 1.10]{crs}).  So, we get a countable (not finite) $(0, \infty)$-minimal conical partition in $\R^{n-1}$. By repeating this argument, we eventually end up at $n=1$, getting a contradiction with the straightforward fact that in $\R$ conical partitions are finite. 
\end{proof} 

\begin{remark}\upshape 
It is an open question which is  the maximal number of chambers of a $(0, \infty)$-minimal conical partition in $\R^n$ for $n>1$. In the case of the local isotropic perimeter, it is known that in $\R^2$, this number is $3$ and in $\R^3$ this number is $4$, as proved by J. Taylor (see \cite{taylor, maggi}). 
\end{remark} 

\subsection{Regularity in the case of  the local perimeter}\ \ \\
In this section we assume that  $\Per$ is   the local anisotropic perimeter, induced by a $G$-periodic norm  $\nu\mapsto \phi(x, \nu)$  on $\R^n$, with such that $\phi^2$ is uniformly convex and $\phi^2\in C^2(\R^n\times \R^n)$. We review well known results about regularity of locally minimal partitions.

\begin{theorem}\label{structuretheorem} 
Let $\{E_k\}_{k\in\mathbb I}$ be a partition of $\R^n$ with finite perimeter.  Then 
\[\mathcal{H}^{n-1} (\partial \{E_k\}_{k\in\mathbb I}\setminus\partial^*\{E_k\}_{k\in\mathbb I})=0. \]
\end{theorem} 
 \begin{proof} The result is a consequence of the structure of Caccioppoli sets (see \cite[Lemma 1.4]{ct1} and \cite[Proposition 2.1]{npst2}). 
   \end{proof} 

\begin{theorem}[$C^{1, \alpha}$ regularity]\label{regthm} 
Let $D\subset\R^n$ be an isoperimetric fundamental domain. 
Then $\partial D$ is a $C^{1,\alpha}$ hypersurface, for some $\alpha\in (0,1)$, up to a closed singular set $\Sigma\neq \emptyset$ 
with $\mathcal{H}^{n-1}(\Sigma)=0$. 
\end{theorem} 

\begin{proof} Let $\{E_i\}_{i\in \mathbb N}$  be the  $G$-periodic partition of $\R^n$ generated by $D$. 
Then by Proposition \ref{partmin} this partition is  $(0,r)$-minimal for every $r<\frac12$.
Fix $x_0$ in the reduced boundary of the partition. Then, up to reordering the indexes, we have that $x_0\in \partial^* E_1\cap \partial^* E_2$. So, there exists $r>0$ sufficiently small such that $|\cup_{i>2}E_i\cap B_r(x_0)| \leq \sigma_0 r^n$ where $\sigma_0=\sigma_0(2,n)$ is as in Lemma \ref{infiltration}. Then by Lemma \ref{infiltration} we get that 
in $|\cup_{i>2}E_i\cap B_{r/2}(x_0)|=0$  and we may apply the classical  regularity theory (see \cite{bombieri}), which gives that 
$\partial D \cap B_{r/2}(x_0)$ is a smooth hypersurface with constant mean curvature, outside a closed singular set of zero $(n-1)$-dimensional Hausdorff measure.  
\end{proof}

In the case of the isotropic perimeter, we can recover the result of \cite{choe} and extend it to every dimension.
\begin{theorem} \label{teoteo}
In the case of the isotropic perimeter, the $G$- periodic partition generated by a minimal fundamental domain is locally finite.

In particular every minimal fundamental domain is bounded. 
Moreover  $\partial D$ is a $C^\infty$ hypersurface in $\R^n$  up to a closed singular set $\Sigma\neq \emptyset$ 
with $\mathcal{H}^{n-1}(\Sigma)=0$. Finally, if $n=2$, then $\Sigma$ is a discrete set. 
\end{theorem}  

 \begin{proof}  The local finiteness of a (locally) minimal partition has been proved in \cite[Theorem 10]{mt}. The main technical part  is to show that the blow-up $L^1$ limit  of a partition is given by a conical partition. This result is obtained as a consequence of a monotonicity formula, see \cite[Lemma 5]{mt}. Once that this result is proved, it is possible to apply Proposition \ref{conicalpart}, which given the finiteness of any $(0,\infty)$-minimal conical partition. From this, using the Elimination Lemma \ref{infiltration}, one concludes the local finiteness of the partition (with the same argument used in the proof of Proposition \ref{conicalpart}).
  
This  implies the boundedness of the isoperimetric fundamental domain, see Proposition \ref{locallyfinite}. Once that the partition is locally finite, we get upper and lower density bounds on the elements of the partitions and so the standard regularity theory applies (see \cite[Theorem IV.2.1, Theorem IB.2.7]{maggi}).
 \end{proof} 
 
\subsection{Regularity in the case of the fractional perimeter}\ \ \\
\\ 
Let now $\Per$ be the fractional perimeter induced by the kernel 
\[
  K_s(h) = \frac{1}{|h|^{n+s}} \quad s\in (0,1).
\]
We shall prove the following analog of Theorem \ref{teoteo}.

 \begin{theorem}\label{regfrac}   
 The periodic partition generated by a minimal fundamental domain is locally finite.
 
In particular every minimal fundamental domain is bounded. 
Moreover  $\partial D$ is a $C^\infty$  hypersurface in $\R^n$ up to a closed singular set $\Sigma\neq \emptyset$ 
with $\mathcal{H}^{n-1}(\Sigma)=0$. Finally, if $n=2$, then $\Sigma$ is a discrete set. 
\end{theorem}  

 \begin{proof}  As in the local case, the main technical part  is to prove that the  $L^1$  limit  of the blow-up of a locally minimal  partition (which exists due to   the estimate in Lemma \ref{estimate}, and by the rescaling  and compactness properties of the fractional perimeter, see the proof of Proposition \ref{conicalpart})  is given by a conical partition. The fact that 
 the $L^1$ limit of the blow-up of a locally minimal partition is conical is a consequence of a monotonicity formula first obtained for the extension problem in \cite{crs} and then generalized to finite partitions in \cite[Theorem 3.10]{cm}. The generalization of this results to the case of countable partitions is straightforward. 
 
As a consequence, it is possible to apply Proposition \ref{conicalpart}, which gives the finiteness of any $(0,\infty)$-minimal conical partition and then  the local finiteness of the initial partition. If the partition is locally finite, we get upper and lower density bounds on the elements of the partitions and  the regularity theory obtained for finite partitions applies (see \cite[Theorem 1.1]{cm}) and we get that $\partial D$ is a $C^{1, \alpha}$ hypersurface in $\R^n$, for some $\alpha\in (0,1)$,  up to a closed singular set $\Sigma\neq \emptyset$ 
with $\mathcal{H}^{n-1}(\Sigma)=0$. and that $n=2$, then $\Sigma$ is a discrete set of points. 

Finally in order to pass from $C^{1,\alpha}$ to $C^\infty$ regularity, we need a bootstrap argument. Let $D$ be a generic fundamental domain, 
and denote $D_i=(D-i)\cap (0,1)^n$  for $i\in\Z^n$.  We denote 
$J(A, B)=\int_A\int_B \frac{1}{|x-y|^{n+s}}$,     and we observe that, due to the fact that $|D\cap (D+k)|=0$ 
for every $k\in\Z^n$ and that $\R^n=\cup_{k\in \Z^n}D+k$, 
there holds for $i,j,k\in \Z^n$, 
\begin{eqnarray*} \Per(D)&=&\sum_{i, j, k\neq j-i} J_s(D_i+i, D_j+i+k)= \sum_{i, j\in \Z^n}\int_{D_i}\int_{D_j}\sum_{k\in \Z^n, k\neq j-i}\frac{1}{|x-y-k|^{n+s}}dxdy\\&=&
\sum_{i, j\in \Z^n}\int_{D_i}\int_{D_j} K_{ij}(x,y)dxdy\qquad\text{ with } K_{ij}(x,y):=\sum_{k\in \Z^n, k\neq j-i}\frac{1}{|x-y-k|^{n+s}}.\end{eqnarray*}
Let us   fix $x\in \partial^* D_i\cap \partial^* D_j$.  Then the first variation of $\Per(D)$ at $x$ is given by
\begin{eqnarray*} H(x, D_i)&=&  \int_{\R^n} \left[\sum_{l}\sum_{k\in \Z^n, k\neq l-i}\ \frac{\chi_{D_l+k}(y)}{|x-y|^{n+s}} -\sum_{l} \sum_{k\in \Z^n, k\neq l-j}\frac{\chi_{D_l+k}(y)}{|x-y|^{n+s}}\right]dxdy  
\\ &= & \int_{\R^n} \left[\sum_{k\in \Z^n, k\neq j-i}\ \frac{\chi_{D_j+k}(y)}{|x-y|^{n+s}} - \sum_{k\in \Z^n, k\neq i-j}\frac{\chi_{D_i+k}(y)}{|x-y|^{n+s}}\right]dy   \\
&+&  \int_{\R^n} \left[\sum_{k\in \Z^n, k\neq 0}\ \frac{\chi_{D_i+k}(y)}{|x-y|^{n+s}}  -\int_{\R^n} \sum_{k\in \Z^n, k\neq 0}\frac{\chi_{D_j+k}(y)}{|x-y|^{n+s}}\right]dy \\
&+& \sum_{l\neq i, j}  \int_{\R^n}  \frac{\chi_{D_l+l-j}(y)}{|x-y|^{n+s}}-\sum_{l\neq i, j} \int_{\R^n}  \frac{\chi_{D_l+l-i}(y)}{|x-y|^{n+s}}dy\\
 &= & \int_{\R^n}  \ \frac{\chi_{D_j}(y)-\chi_{D_i}(y)}{|x-y|^{n+s}}dy   \\
&+&  \int_{\R^n}  \frac{\chi_{D_i+i-j}(y)}{|x-y|^{n+s}}dy-\int_{\R^n}  \frac{\chi_{D_j+j-i}(y)}{|x-y|^{n+s}}dy\\
&+& \sum_{l\neq i, j}  \int_{\R^n}  \frac{\chi_{D_l+l-j}(y)}{|x-y|^{n+s}}dy-\sum_{l\neq i, j} \int_{\R^n}  \frac{\chi_{D_l+l-i}(y)}{|x-y|^{n+s}}dy.
 \end{eqnarray*} 
If $D$ is an isoperimetric fundamental domain,  the equilibrium condition for $x\in \partial^*D_i\cap \partial^* D_j$ reads \[H(x, D_i)=H(x, D_j).\]  Reasoning as in  \cite[Theorem 5.1]{crs} (see also \cite[Theorem 2.6]{cn})  we get that $D$ satisfies in the viscosity sense  
\begin{eqnarray}\label{el} \int_{\R^n}  \ \frac{\chi_{D_j}(y)-\chi_{D_i}(y)}{|x-y|^{n+s}}dy = f(x, D)&:=&  \int_{\R^n}  \frac{\chi_{D_j+j-i}(y)- \chi_{D_i+i-j}(y)}{|x-y|^{n+s}}dy\\ & +& \sum_{l\neq i, j}  \int_{\R^n}  \frac{\chi_{D_l+l-i}(y)-\chi_{D_l+l-j}(y)}{|x-y|^{n+s}}dy\nonumber \end{eqnarray}  for all $x\in \partial^*D_i\cap \partial^* D_j$. 
  Note that since $i\neq j$, $l\neq i,j$  and $x\in \partial^* D_i\cap \partial^* D_j$ then  $|x-y|>1$ for $y\in D_j+j-i, D_i+i-j, D_l+l-i, D_l+l-j$ and then $f(\cdot,D)\in C^\infty(B_r(x))$, where $r>0$ is such that $B_r(x)$ does not contain singular points of $\partial D$. So, we may apply the bootstrap argument in \cite[Theorem 1.6]{bfv}  (see also \cite[Theorem 2.6]{cn}) to obtain the desired $C^\infty$ regularity. 
\end{proof} 

We conclude with a straightforward consequence of the previous result, about existence of fractional minimal cones. 
\begin{corollary}There exists a conical partition  of $\R^n$ with at least three phases which is locally minimal for the fractional perimeter. 
\end{corollary} 
\begin{proof} Let us reduce to the case of $\R^2$. Indeed if $\{C_i\}_{i\in\mathbb{I}}$ is a locally minimal conical partition in $\R^2$, then $\{C_i\times \R^{n-2}\}_{i\in\mathbb{I}}$ is a locally minimal conical partition
in $\R^n$ for $n>2$. 

By Theorem \ref{regfrac} every isoperimetric fundamental domain is bounded and smooth, up to a finite number of singular points. 
In order to conclude it is sufficient to show that every isoperimetric fundamental domain $D$ has at least one singular point: if it is the case, the $L^1$ limit of the blow-up of the partition generated by $D$ at one of this singular points is a locally minimal conical partition with at least three phases. 

Assume by contradiction that there exists an isoperimetric fundamental domain $D$ which has no singular points. Let $\partial D=\cup_{i=1}^N \gamma_i$, where each  $\gamma_i$ is a Jordan curve and $N\geq 1$. So, $D$ is the union of $M\leq N$ bounded connected   components. For every $i=1,\dots, N$, there exists  an integer  translation $D+k_i$ of $D$  such that $\gamma_i\in \partial (D+k_i)$.   Let us take the connected component $D_i$ of $D$ with biggest diameter and let $\gamma_i$ its exterior boundary (i.e. the boundary of the unbounded component of the complement of $D_i$). Then there exists an integer  translation $D+k_i$ of $D$  such that 
$\gamma_i\in \partial (D+k_i)$. But then at least one connected component of $D+k_i$ would have diameter bigger than $D_i$, giving a contradiction.  \end{proof} 
 
 \section{Anisotropic minimal partitions of the plane} \label{secplane}
 
In this section we reduce to $X=\R^2$, equipped with the Lebesgue measure and with the Euclidean distance, and we fix $G= \Z^2$ 
(but the same discussion applies to any discrete group of translations).
We shall also assume that $\Per$ is   the anisotropic perimeter induced by a  spatially homogeneous norm $\phi$ on $\R^2$,   that is,
\[\Per(E)=\int_{\partial^* E} \phi(\nu(x))dH^1(x).\]
 We will denote by $\phi^*$ the dual of $\phi$, that is
$\phi^*(x)=\sup\{x\cdot y\ :\ \phi(y)\leq 1\}$, and $W_\phi=\{x\in\R^2, \phi^*(x)\leq 1\}$ will be the Wulff shape. 

\begin{proposition} \label{probounded}
Every isoperimetric fundamental domain $D$ is bounded and satisfies
\[
{\rm diam}(D)\le \sqrt 2 + \frac{\Per(D)}{2}.
\]
\end{proposition}

\begin{proof} 
Let $D$ be an isoperimetric fundamental domain. Since $D$ has finite perimeter,
by \cite[Theorem 1]{acmm}  it can be decomposed in a finite or countable family of {\it indecomposable components} $(D_i)_i$ such that 
$|D_i\cap D_j|=0$ for $i\ne j$, and $\Per(D)=\sum_i \Per(D_i)$. Up to choosing a suitable representative for the components, we can assume that 
$D_i\cap D_j=\emptyset$ for $i\ne j$, which in turn implies that 
\[
D_i\cap (D_j+z)=\emptyset \qquad \text{for $i\ne j$ and for $z\in\mathbb Z^2$.}
\]
In particular, for all $i$ there exists $z_i\in \mathbb Z^2$ such that $\widetilde D_i = D_i+z_i$ intersects $[0,1]^2$,
and $\widetilde D=\cup_i \widetilde D_i$ is still an isoperimetric fundamental domain.

By \cite[Lemma 2.13]{dmnp} we also have that each component $D_i$ is bounded and satisfies
\[
{\rm diam}(D_i)={\rm diam}(\widetilde D_i)\le \frac{\Per(D_i)}{2}\le \frac{\Per(D)}{2},
\]
which gives that
\[
{\rm diam}(\widetilde D)\le {\rm diam}([0,1]^2) + \frac{\Per(D)}{2} =  \sqrt 2 + \frac{\Per(D)}{2}.
\]
\end{proof} 

\begin{proposition} \label{proregular}
The boundary of an isoperimetric fundamental domain is composed by a finite number of Lipschitz {\it edges}, 
which minimize the anisotropic length, joining a finite number of {\it vertices} where up to four edges may concur.
\end{proposition}

\begin{proof} 
By Proposition \ref{locallyfinite}, the $\mathbb Z^2$-periodic partition induced by $D$ is locally finite. 
As a consequence, by \cite[Theorem 4.1]{morgan} such partition is composed by a family of Lipschitz curves meeting at singular points
which are locally finite. Moreover each curve is a Lipschitz graph locally minimizing the anisotropic length, 
and at each singular point do concur three or four curves. 
\end{proof} 

\begin{proposition} \label{proconvex}
Assume that  $(\phi)^2$ is strictly convex. Then every isoperimetric fundamental domain is either a centrally symmetric convex hexagon or parallelogram. 
\end{proposition}

\begin{proof} 
Recalling Proposition \ref{proregular}, we observe that the boundary of an isoperimetric fundamental domain $D$ is composed by a finite number of segments, joining a finite number of vertices. 

We claim that $D$ is a convex polygon. Indeed, letting $v$ be a vertex of $D$,
the periodic minimal partition induced by $D$ in $B_r(v)$, with $r>0$ small enough, is given by segments with one endpoint in $v$ and the other on $\partial B_r(v)$. The number of such segments is less than or equal to four, and the number of components of ${\rm int}(D) \cap B_r(v)$ is one or two.
Notice that, if we replace ${\rm int}(D) \cap B_r(v)$ with its convex envelope, the perimeter of the partition decreases, it follows that 
$D \cap B_r(v)$ is a convex circular sector, and the same applies to the other regions of the partition. It follows that 
the number of such regions is three or four, and that $D$ is a convex polygon, as claimed.

Finally, a classical result by Fedorov \cite{fedorov} gives that a planar fundamental domain which is also a convex polygon is necessarily a centrally symmetric hexagon or parallelogram.
\end{proof} 
 
 By approximating a general norm with differentiable norms, from Proposition \ref{proconvex}
 we obtain the following result.
 
 \begin{proposition} \label{procongen}
 For any norm $\p$, there exists an isoperimetric fundamental domain 
given by a centrally symmetric convex hexagon or parallelogram. 
\end{proposition} 

 \begin{remark}\upshape 
 Even if the isoperimetric fundamental domain in Proposition \ref{procongen} might be nonunique for a general $\p$, if the Wulff shape $W_\p$ is
 a hexagon or a parallelogram which tessellate the plane, then the only isoperimetric fundamental domain is given by $W_\p/|W_\p|$. This follows from the fact that the Wulff shape is the unique volume-constrained minimizer, up to translations, of the anisotropic perimeter (see \cites{fm,maggi}).
 \end{remark}
 
 We conclude by noticing that, under some assumptions on the norm $\p$, we can exclude parallelograms as possible minimizers.
 
\begin{proposition} 
If there $\phi^2$  is strictly convex and differentiable, then the only isoperimetric fundamental domains are hexagons.
\end{proposition} 

\begin{proof} To prove that the isoperimetric fundamental domains are hexagons, it is sufficient to show that crosses are not locally minimal for the perimeter. 

For the proof we refer  to the following  figure and we show that it is not possible that both  the perimeter of the curve $EO\cup OC $ is bigger than the perimeter of the   curve $ED\cup DO\cup DC$ and the perimeter of the curve $BO\cup OC $ is bigger than the perimeter of the curve
 $OA\cup AC\cup AB$. We argue by contradiction. 
 
 Let us call $\alpha$ the angles $BOA=BOC$, and let us fix $OC=OB=1$ and $OA=c<1$. \\
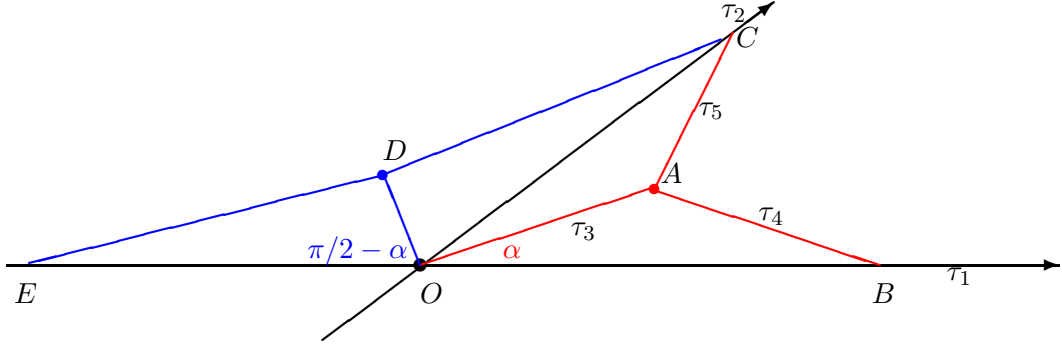
\begin{picture}(8,5)(-2,-1) 
\thicklines
\put(-1.5,0){\vector(1,0){14}}
\put(11,-0.2){$\tau_1$}
\put(2.7,-1){\vector(4,3){6}}
\put(4,0){\circle*{0.18}}
\put(4,-0.5){$O$}
\put(8,3.3){$\tau_2$}
\put(4,0){\color{red}\line(3,1){3.15}}
\put(6,0.4){$\tau_3$}
\put(5.1,0.1){\color{red}$\alpha$}
\put(7.11,1){\color{red}\line(1,2){1.05}}
\put(7.11,1){\color{red}\line(3,-1){3}}
\put(7.7, 2){$\tau_5$}
\put(7.11,1.01){\color{red}\circle*{0.15}}
\put(7.2,1.1){$A$}
\put(8.5, 0.6){$\tau_4$}
\put(4,0){\color{blue}\line(-2,5){0.5}}
\put(2.5,0.1){\color{blue}$\pi/2-\alpha$}
\put(3.5,1.2){\color{blue}\line(5,2){4.5}}
\put(3.5,1.2){\color{blue}\circle*{0.15}}
\put(3.5,1.4){$D$}
\put(3.5,1.2){\color{blue}\line(-4,-1){4.7}}
\put(8.2,2.9){$C$}
\put(10,-0.5){$B$}
\put(-1.4,-0.5){$E$}
\end{picture}

Let us denote $\nu_i$ one of the two   vectors such that $\nu_i\cdot \tau_i=0$,  $|\nu_i|=1$.  
We have that \begin{eqnarray*} \nu_1&=&(0,1)\\ \nu_2&=&(-\sin 2\alpha, \cos 2\alpha)\\\nu_3&=&\frac{\nu_1+\nu_2}{|\nu_1+\nu_2|}=\frac{1}{2\cos\alpha} (-\sin 2\alpha, \cos 2\alpha+1)\\\nu_4&=&\nu_4(c)=\frac{1}{\sqrt{1+c^2-2c\cos\alpha}}(c\sin\alpha, 1-c\cos \alpha)\\\nu_5&=&\nu_5(c)=\frac{1}{\sqrt{1+c^2-2c\cos\alpha}}[(1-c\cos \alpha)\nu_2+c\sin\alpha \tau_2]\end{eqnarray*}  where we used the fact that $\sqrt{2+2\cos 2\alpha}=2\cos\alpha$. 
 For $c>0$ sufficiently small we have that
 \begin{eqnarray*} \phi(\nu_4(c))&=&\phi(\nu_1)+c\sin\alpha \nabla \phi(\nu_1)\cdot \tau_1\\ \phi(\nu_5(c))&=&\phi(\nu_2)+c\sin\alpha \nabla \phi(\nu_2)\cdot \tau_2\end{eqnarray*}
 Then the perimeter of the curve $OA\cup AC\cup AB$ is given by
 \begin{eqnarray*} \Per(OACB) &= & c\phi(\nu_3)+\sqrt{1+c^2-2c\cos \alpha}\phi(\nu_4)+ \sqrt{1+c^2-2c\cos \alpha}\phi(\nu_5)\\
 &=&\frac{ c}{2\cos\alpha} \phi(\nu_1+\nu_2) +\sqrt{1+c^2-2c\cos \alpha}[\phi(\nu_1)+\phi(\nu_2)]\\ 
 && +c\sin\alpha \sqrt{1+c^2-2c\cos \alpha}[\nabla \phi(\nu_1)\cdot \tau_1+\nabla \phi(\nu_2)\cdot \tau_2].  \end{eqnarray*} 
 Assume that the perimeter of the curve $BO\cup OC $ is bigger than the perimeter of the curve
 $OA\cup AC\cup AB$: this means that
  \begin{eqnarray*} \phi(\nu_1)+\phi(\nu_2)&=& \Per(BOC)\leq \Per(OACB) \\ 
 &=&\frac{ c}{2\cos\alpha} \phi(\nu_1+\nu_2) +\sqrt{1+c^2-2c\cos \alpha}[\phi(\nu_1)+\phi(\nu_2)]\\ 
 && +c\sin\alpha \sqrt{1+c^2-2c\cos \alpha}[\nabla \phi(\nu_1)\cdot \tau_1+\nabla \phi(\nu_2)\cdot \tau_2].  \end{eqnarray*} 
 This implies
  \begin{eqnarray*}  && \frac{2\cos \alpha[1-\sqrt{1+c^2-2c\cos \alpha}]}{c}[\phi(\nu_1)+\phi(\nu_2)]\\ &\leq&\phi(\nu_1+\nu_2)+\sin(2\alpha)  \sqrt{1+c^2-2c\cos \alpha}[\nabla \phi(\nu_1)\cdot \tau_1+\nabla \phi(\nu_2)\cdot \tau_2]\nonumber \end{eqnarray*}
  and sending $c\to 0$
   \begin{equation}\label{uno} 2\cos^2\alpha [ \phi(\nu_1)+\phi(\nu_2)]\leq\phi(\nu_1+\nu_2)+\sin(2\alpha)  [\nabla \phi(\nu_1)\cdot \tau_1+\nabla \phi(\nu_2)\cdot \tau_2]. \end{equation}

Analogously, on the other side, we have (substituting $\alpha$ with $\pi/2-\alpha$, $\nu_2$ with $-\nu_2$  and $\tau_1$ with $-\tau_1$).
\begin{eqnarray*} \Per(ODCE) &= & \frac{ c}{2\sin\alpha} \phi(-\nu_1+\nu_2) +\sqrt{1+c^2-2c\sin \alpha}[\phi(\nu_1)+\phi(\nu_2)]\\ 
 && +c\cos\alpha \sqrt{1+c^2-2c\sin \alpha}[\nabla \phi(\nu_1)\cdot (-\tau_1)+\nabla \phi(-\nu_2)\cdot \tau_2].  \end{eqnarray*} 
 Again, reasoning as above,  assuming that  the perimeter of the curve $EO\cup OC $ is bigger than the perimeter of the curve
 $OD\cup DE\cup DC$ and sending $c\to 0$, we get 
    \begin{equation}\label{due} 2\sin^2\alpha [ \phi(\nu_1)+\phi(\nu_2)]\leq\phi(-\nu_1+\nu_2)-\sin(2\alpha)  [\nabla \phi(\nu_1)\cdot \tau_1+\nabla \phi(\nu_2)\cdot \tau_2]. \end{equation}

Summing up \eqref{uno} and \eqref{due} we obtain
\[ 2 [ \phi(\nu_1)+\phi(\nu_2)]\leq\phi(-\nu_1+\nu_2)+\phi(\nu_1+\nu_2) \] 
which contradicts the strict convexity of $\phi$. 
 \end{proof}


 \end{document}